\documentclass{amsart}
\usepackage{amsmath}
\usepackage{amssymb}
\usepackage{amsfonts} 

\usepackage[hypertex,colorlinks=true,linkcolor=blue,citecolor=blue,]{hyperref}

\relax
\PassOptionsToClass{twoside}{article}

\RequirePackage{amsfonts,amssymb,amsmath,amscd,amsthm}
\RequirePackage{txfonts}
\RequirePackage{graphicx}
\RequirePackage{xcolor}
\RequirePackage{enumerate}

\newtheorem{theorem}{Theorem}[section]
\newtheorem{lemma}[theorem]{Lemma}
\newtheorem{corollary}[theorem]{Corollary}

\theoremstyle{definition}
\newtheorem{definition}[theorem]{Definition}

\newtheorem{example}[theorem]{Example}

\theoremstyle{remark}
\newtheorem{remark}[theorem]{Remark}

\numberwithin{equation}{section}

\begin{document}

\newcommand{\om}{\omega}
\newcommand{\si}{\sigma}
\newcommand{\la}{\lambda}
\newcommand{\ph}{\varphi}
\newcommand{\ep}{\varepsilon}
\newcommand{\tep}{\widetilde{\varepsilon}}
\newcommand{\al}{\alpha}
\newcommand{\sub}{\subseteq}
\newcommand{\hra}{\hookrightarrow}
\newcommand{\de}{\delta}

\
\newcommand{\RR}{{\mathbb{R}}}
\newcommand{\NN}{{\mathbb{N}}}
\newcommand{\DD}{{\mathbb{D}}}
\newcommand{\CC}{{\mathbb{C}}}
\newcommand{\ZZ}{{\mathbb{Z}}}
\newcommand{\T}{{\mathbb{T}}}
\newcommand{\TT}{{\mathbb{T}}}
\newcommand{\KK}{{\mathbb{K}}}
\newcommand{\px}{\partial X}
\newcommand{\cL}{{\mathcal{L}}}
\newcommand{\cA}{{\mathcal{A}}}
\newcommand{\cB}{{\mathcal{B}}}

\def\C{\mathbb C}
\def\R{\mathbb R}
\def\b{\mathcal B}
\def\c{\mathcal C}
\def\X{\mathbb X}
\def\U{\mathcal U}\def\M{\mathcal M}
\def\a{\mathcal A}
\def\K{\mathbb K}
\def\F{\mathbb F}
\def\b{\mathfrak{B}}
\def\f{\mathfrak{F}}
\def\x{\mathfrak{X}}
\def\Z{\mathbb Z}
\def\P{\mathbb P}
\def\v{\vartheta_\alpha}
\def\va{{\varpi}_\alpha}
\def\I{\mathbb I}
\def\H{\mathbb H}
\def\Y{\mathbb Y}
\def\E{\mathbb E}
\def\N{\mathbb N}
\def\cal{\mathcal}

\title[Existence of Pseudo-Almost Automorphic Solutions]{Existence Results for Some Damped Second-Order Volterra Integro-Differential Equations}

\author{Toka Diagana}
\address{Department of Mathematics, Howard University, 2441 6th
Street N.W.,  Washington, D.C. 20059, USA}

\email{tdiagana@howard.edu}



\maketitle

\begin{center}
{\it In Memory of Prof. Yahya Ould Hamidoune}
\end{center}

\begin{abstract} In this paper we make a subtle use of operator theory techniques and the well-known Schauder fixed-point principle to establish the existence of pseudo-almost automorphic solutions to some second-order damped integro-differential equations with pseudo-almost automorphic coefficients.  In order to illustrate our main results, we will study the existence of pseudo-almost automorphic solutions to a structurally
damped plate-like boundary value problem.
\end{abstract}

\maketitle

\section{Introduction} Integro-differential equations play an important role when it comes to modeling various natural phenomena, see, e.g.,
 \cite{CC, C, da2, F1, F2, HR, G, J, M, NO, N, pruss, V1, V2, V3, webb}.
In recent years, noteworthy progress has been made in studying the existence of periodic, almost periodic, almost automorphic, pseudo-almost periodic, and pseudo-almost automorphic solutions to first-order integro-differential equations, see, e.g., \cite{AG, TDbook, DE, DHH, DE1, DE2, H0, H1, H2, li1, li2, V1, V2, V3}. 
The most popular method used to deal with the existence of solutions to those 
first-order integro-differential equations consists of the so-called method of resolvents, see, e.g., \cite{adez, da1, da2, H1, H2, li1, li2}. 

Fix $\alpha \in (0, 1)$. Let $\H$ be an infinite dimensional separable Hilbert space over the field of complex
numbers equipped with the inner product and norm given respectively by $\langle \cdot, \cdot \rangle$ and $\|\cdot\|$.
The purpose of this paper consists of making use of a new approach to study the existence of pseudo-almost automorphic solutions to the class of damped second-order Volterra integro-differential equations given by
\begin{eqnarray}\label{PR}
\frac{d^2 \ph}{dt^2} + B \frac{d\ph}{dt}+ A \ph  = \int_{-\infty}^t C(t-s) \ph (s) ds + f(t,\ph),
\end{eqnarray}
where $A: D(A) \subset \H \mapsto \H$ is 
an unbounded self-adjoint linear operator whose spectrum consists of isolated eigenvalues given by
$$0 < \lambda_1 < \lambda_2 < ...< \lambda_n \to \infty$$ as $n \to \infty$ with each eigenvalue having a finite multiplicity $\gamma_j$ equals to the multiplicity of the corresponding eigenspace, $B: D(B) \subset \H \mapsto \H$ is a positive self-adjoint linear operator such that there exist two constants $\gamma_1, \gamma_2 > 0$ and such that $\gamma_1 A^\alpha \leq B \leq \gamma_2 A^\alpha$, that is,
$$\gamma_1 \langle A^\alpha \ph, \ph \rangle \leq \langle B\ph, \ph \rangle \leq \gamma_2 \langle A^\alpha \ph, \ph \rangle$$
for all $\ph \in D(B^{\frac{1}{2}}) = D(A^{\frac{\alpha}{2}})$,
the mappings $C(t): D(A) \subset \H \mapsto \H$ consist of (possibly unbounded) linear operators for each $t \in \R$, and 
the function $f: \R \times \H \mapsto \H$ is pseudo-almost automorphic in the first variable uniformly in the second one.

Equations of type Eq. (\ref{PR}) 
arise very often in the study of natural phenomena in which a certain memory effect is taken into consideration, see, e.g., \cite{AL, BE, LO, MU, MU1}.
In \cite{AL, BE} for instance, equations of type Eq. (\ref{PR}) appeared in the study of a viscoelastic wave equation with memory. 

The existence, uniqueness, and asymptotic behavior of solutions to Eq. (\ref{PR}) have widely been studied, see, e.g., \cite{AL, BE, BU, CA, LIA1, LIA2, LO, MU, MU1, O1, O2, O3}. However, to the best of our knowledge, the existence of pseudo-almost automorphic solutions to Eq. (\ref{PR}) is an untreated original problem with important applications, which constitutes the main motivation of this paper.

In this paper, we are interested in the special case $B = 2\gamma A^{\alpha}$ where $\gamma > 0$ is a constant, that is,
\begin{eqnarray}\label{PRRR}
\frac{d^2 \ph}{dt^2} + 2\gamma A^\alpha \frac{d\ph}{dt}+ A \ph  = \int_{-\infty}^t C(t-s) \ph (s) ds + f(t,\ph), \ \ t \in \R.
\end{eqnarray}
It should be mentioned that various versions of Eq. (\ref{PRRR}) have been investigated in the literature, see, e.g., Chen and Triggiani \cite{CT2, CT}, Huang \cite{HH1, HH2, HH3, HH4}, and Xiao and Liang \cite{XL, XL1, XL2, XL3, XL4, XL5}.

Consider the polynomial $Q_n^\gamma$ associated with the left hand side of Eq. (\ref{PRRR}), that is,
\begin{eqnarray}\label{D11}
Q_n^\gamma (\rho) := \rho^2 + 2\gamma \lambda_{n}^{\alpha} \rho + \lambda_n\end{eqnarray} and denote its roots by $\rho_{1}^n:= d_n + i e_n$ and $\rho_{2}^n := r_n + i s_n$ for all $n \geq 1$.

In the rest of the paper, we suppose that the roots $\rho_{1}^n$ and $\rho_{2}^n$ satisfy: $\rho_{1}^n \not= \rho_{2}^n$ for all $n \geq 1$ and that the following crucial assumption holds:
there exists $\delta_0 > 0$ such that
\begin{eqnarray}\label{AS1}
\sup_{n\geq 1} \Big[ \max(d_n, r_n)\Big] \leq -\delta_0 < 0.
\end{eqnarray}

In order to investigate the existence of pseudo-almost automorphic solutions to Eq. (\ref{PRRR}), our strategy consists of rewriting it as a first-order integro-differential equation in the product space $\E_{\frac{1}{2}}:=D(A^{\frac{1}{2}}) \times \H$ and then study the existence of pseudo-almost automorphic solutions to the obtained first-order integro-differential equation with the help of Schauder fixed point principle and then go back to Eq. (\ref{PRRR}). 

Recall that the inner product of $\E_{\frac{1}{2}}$ is defined as follows:
$$\Bigg (\left(\begin{smallmatrix}\displaystyle \ph_1\\ \\ \\ \displaystyle \ph_2\end{smallmatrix}\right), \left(\begin{smallmatrix}\displaystyle \psi_1 \\ \\ \\  \displaystyle \psi_2\end{smallmatrix}\right)     \Bigg)_{\E_{\frac{1}{2}}} := \langle A^{\frac{1}{2}} \ph_1, A^{\frac{1}{2}} \psi_1 \rangle + \langle \ph_2, \psi_2 \rangle$$
for all $\ph_1, \psi_1 \in D(A^{\frac{1}{2}})$ and $\ph_2, \psi_2 \in \H$. Its corresponding norm will be denoted $\|\cdot\|_{\E_{\frac{1}{2}}}$.

Letting
$$\displaystyle \Phi:=\left(\begin{smallmatrix}\displaystyle \ph \\ \\ \\ \\ \\ \\ \displaystyle \ph'\end{smallmatrix}\right) \in \E_{\frac{1}{2}},$$
then Eq. \eqref{PRRR}
can be rewritten in the following form
\begin{equation}\label{PRR1}
\frac{d\Phi}{dt} = \a  \Phi + \int_{-\infty}^t \c(t-s) \Phi (s) ds + F(t, \Phi(t)), \;t\in \R,
 \end{equation}
where $\a, \c$ are the operator matrices defined by
\begin{equation*}\label{Therm}
\a  = \left (
\begin{matrix}
0 & I \\ \\ \\
-A & -A^\alpha\\
\end{matrix}
\right), \ \ \ \ 
\c  = \left (
\begin{matrix}
C \\ \\ \\
0 \\
\end{matrix}
\right),
\end{equation*}
with domain $D(\a) = D(A) \times [D(A^{\frac{1}{2}}) \cap D(A^\alpha)] = D(\c)$ ($D(\a) = D(A) \times D(A^{\frac{1}{2}})$ if $0 <\alpha \leq \frac{1}{2}$ and $D(\a) = D(A) \times D(A^\alpha)$ if $\frac{1}{2} \leq \alpha <1$), and  
the function $F: \R \times \E_{\frac{1}{2}} \mapsto \E:= \H \times \H$ is given by 
$$F(t, \Phi)=\left(\begin{smallmatrix}\displaystyle 0\\ \\ \\ \\ \\ \displaystyle f(t,\ph)\end{smallmatrix}\right).$$

In order to investigate Eq. (\ref{PRR1}), we study the first-order differential equation in the space $\E_{\frac{1}{2}}$ given by,
\begin{eqnarray}\label{OPR1}
\frac{d\ph}{dt} = \left(\a + \b\right) \ph   + F(t,\ph), \ \ t\in \R,
\end{eqnarray}
where $\b: C(\R, D(\a)) \mapsto \E_{\frac{1}{2}}$ is the linear operator defined by 
\begin{eqnarray}\label{PR2}\b \ph  := \int_{-\infty}^t \c(t-s) \ph (s) ds, \ \ \ph \in C(\R, D(\a))\end{eqnarray}with $C(\R, D(\a))$ being the collection of
all continuous functions from $\R$ into $D(\a)$.

In order to study the existence of solutions to Eq. (\ref{OPR1}), we will make extensive use of hyperbolic semigroup tools and fractional powers of operators, and that the linear operator $\b$ satisfies some additional assumptions. Our existence result will then be obtained through the use of the well-known Schauder fixed-point theorem. Obviously, once we establish the sought existence results for Eq. (\ref{OPR1}), then we can easily go back to Eq. (\ref{PRRR}) notably through Eq. (\ref{PR2}).

The concept of pseudo almost automorphy is a powerful notion introduced in the literature by Liang {\it et al.} 
\cite{L, LLL, LL, XJ}. This concept has recently generated several developments and extensions, which have been summarized in a new book by Diagana \cite{TDbook}. 
The existence of almost periodic and asymptotically
almost periodic solutions to integro-differential equations of the form 
Eq. (\ref{PRR1}) in a general context has recently been established in \cite{H1, H2}. Similarly, in \cite{li2}, the existence of pseudo-almost automorphic solutions to Eq. (\ref{PRR1}) was studied. The main method used in the above-mentioned papers are resolvents operators.
However, to the best of our knowledge, the existence of
pseudo-almost automorphic solutions to Eq. (\ref{PRRR}) is an important untreated topic with some interesting
applications. Among other things, we will make
extensive use of the Schauder fixed point to
derive some sufficient conditions for the existence of
pseudo-almost automorphic (mild) solutions to \eqref{OPR1} and then to Eq. (\ref{PRRR}).

\section{Preliminaries} Some of the basic results discussed in this section are mainly taken from the following recent papers by Diagana \cite{MCM, PAMS}.
In this paper, $\H$ denote an infinite dimensional separable Hilbert space over the field of complex
numbers equipped with the inner product and norm given respectively by $\langle \cdot, \cdot \rangle$ and $\|\cdot\|$. 
If $A$ is a linear operator upon a Banach space $(\X, \|\cdot\|)$, then the notations $D(A)$, $\rho(A)$,
$\sigma(A)$, $N(A)$, and $R(A)$ stand respectively for the domain, resolvent, spectrum, kernel, and the range of $A$.
Similarly, if $A: D:=D(A) \subset \X \mapsto \X$ is a closed linear operator on a Banach space, one denotes its graph norm by $\|\cdot\|_{D}$ defined by
$\|x\|_D := \|x\| + \|Ax\|$ for all $x \in D$. 
From the closedness of $A$, one can easily see that $(D, \|\cdot\|_{D})$ is a Banach space.
Moreover, one sets $R(\lambda, L) := (\lambda I - L)^{-1}$ for all
$\lambda \in \rho(A)$. 
We set $Q =I-P$ for a projection
$P$. If $\Y, \Z$ are Banach spaces, then the space $B(\Y, \Z)$ denotes the collection of all bounded
linear operators from $\Y$ into $\Z$ equipped with its natural
uniform operator topology $\|\cdot\|_{B(\Y, \Z)}$. We also set $B(\Y) = B(\Y, \Y)$.
If $K \subset \X$ is a subset, we let $\overline{co}\, K$ denote the closed convex hull of $K$. Additionally, 
$\T$ will denote the set defined by, $\T :=\{(t,s) \in \R \times \R: t \geq s\}.$
If $(\X, \|\cdot\|_\X)$ and $(\Y, \|\cdot\|_\Y)$ are Banach spaces, their product $\X \times \Y :=\{(x,y): x \in \X, \ y \in \Y\}$ is also a Banach when it is equipped with the norm
given by $$\|(x,y)\|_{\X \times \Y} = \sqrt{\|x\|_{\X}^2 + \|y\|_{\Y}^2} \ \ \mbox{for all} \ \ (x,y) \in \X \times \Y.$$

In this paper if $\beta \geq 0$, then we set $\E_\beta := D(A^\beta) \times \H$, and $\E := \H \times \H$ and equip them with their corresponding topologies $\|\cdot\|_{\E_\beta}$ and $\|\cdot\|_\E$. Recall that $D(A^\beta)$ will be equipped
with the norm defined by, $\|\ph\|_\beta := \|A^\beta \ph\|$ for all $\ph \in D(A^\beta)$.

In the sequel, $A: D(A) \subset \H \mapsto \H$ stands for a self-adjoint (possibly unbounded) linear operator on the Hilbert space $\H$ whose spectrum consists of isolated eigenvalues $0 < \lambda_1 < \lambda_2 < ...< \lambda_n \to \infty$ with each eigenvalue having a finite multiplicity $\gamma_j$ equals to the multiplicity of the corresponding eigenspace.
Let $\{e_{j}^k\}$ be a (complete) orthonormal sequence of eigenvectors associated with the eigenvalues $\{\lambda_j\}_{j\geq 1}$.
Clearly, for each $u \in D(A)$, where if $$\displaystyle u \in D(A) :=\Big\{u \in \H: \quad \sum_{j=1}^\infty \lambda_j^2 \| E_j u\|^2 < \infty\Big\}, \ \ \mbox{then} \ \
Au = \sum_{j=1}^\infty \lambda_j \sum_{k=1}^{\gamma_j} \langle u, e_{j}^k \rangle e_{j}^k = \sum_{j=1}^\infty \lambda_j E_j u$$
with $E_j u =\sum_{k=1}^{\gamma_j} \langle u, e_{j}^k \rangle e_{j}^k.$
Note that $\{E_j\}_{j\geq1}$ is a sequence of orthogonal projections on $\H$. Moreover, each $u \in \H$ can written as follows:
$u = \sum_{j=1}^\infty E_j u.$
It should also be mentioned that the operator $-A$ is the infinitesimal generator of an analytic semigroup $\{S(t)\}_{t \geq 0}$, which is explicitly expressed in terms of those orthogonal projections $E_j$ by, for all $u \in \H$,
$$S(t) u = \sum_{j=1}^\infty e^{-\lambda_j t} E_j u$$which in particular is exponentially stable as $$\|S(t)\| \leq e^{-\lambda_1 t}$$ for all $t \geq 0$.

\section{Sectorial Linear Operators}
The basic results discussed in this section are mainly taken from Diagana \cite{TDbook, d6}.

\begin{definition}{\rm \label{sect}
A linear operator $B: D(B) \subset \X \mapsto \X$ (not necessarily
densely defined) on a Banach space $\X$ is said to be sectorial if the following hold:
there exist constants $\omega\in \R$, $\displaystyle \theta\in
\left(\frac{\pi}{2},\pi\right)$, and $M>0$ such that $\rho(B)\supset
S_{\theta,\omega}$,
\begin{align}\label{sect}
&S_{\theta,\omega}:=\Big\{\lambda\in \C:\lambda\neq\omega,\ \ |\arg(\lambda-\omega)|<\theta\Big\}, \ \ \mbox{and}\\
&\|R(\lambda, B)\|\leq \frac{M}{|\lambda-\omega|},
\quad \lambda\in S_{\theta,\omega}.
\end{align}}
\end{definition}

\begin{example}{\rm
Let $p \geq 1$ and let $\Omega \subset \R^d$ be open bounded
subset with $C^2$ boundary $\partial \Omega$. Let $\X:=
L^p(\Omega)$ be the Lebesgue space equipped with the norm,
$\|\cdot\|_p$ defined by,
$$\|\varphi\|_p = \Big(\int_{\Omega} |\varphi(x)|^p dx\Big)^{1/p}.$$

Define the operator $A$ as follows:
$$D(B)= W^{2, p}(\Omega) \cap W^{1,p}_{0}(\Omega), \ \ B(\varphi)= \Delta \varphi, \ \ \forall \varphi \in D(B),$$
where $\displaystyle \Delta = \sum_{k=1}^d
\frac{\partial^2}{\partial x_k^2}$ is the Laplace operator.
It can be checked that the operator $B$ is sectorial on
$L^p(\Omega)$. }
\end{example}

It is well-known \cite{Lun} that if $B: D(B) \subset \X \mapsto \X$ is a sectorial linear operator, then
it generates an analytic semigroup $(T(t))_{t\geq0}$, which maps
$(0,\infty)$ into $B(\X)$ and such that there exist $M_0, M_1 > 0$
with
\begin{align}
&\|T(t)\|\leq M_0 e^{\omega t}, \quad t> 0,\\
&\|t(A-\omega)T(t)\|\leq M_1 e^{\omega t}, \quad t>
0.\label{analy}
\end{align}

In this paper, we suppose that the semigroup
$(T(t))_{t\geq0}$ is hyperbolic, that is, there exist a projection
$P$ and constants $M, \delta>0$ such that $T(t)$ commutes with
$P$, $N(P)$ is invariant with respect to $T(t)$, $T(t):
R(Q)\mapsto R(Q)$ is invertible, and the following hold
\begin{equation}\label{hyP}
\|T(t)Px\|\leq Me^{-\delta t}\|x\| \qquad \mbox{ for } t\geq0,
\end{equation}
 \begin{equation}\label{hyQ}
 \|T(t)Qx\|\leq
Me^{\delta t}\|x\| \qquad \mbox{ for } t\leq0, \end{equation}
 where $Q:=I-P$ and,  for $ t\leq0$,  $T(t):= (T(-t))^{-1} $.

Recall that the analytic semigroup $(T(t))_{t\geq0}$ associated
with $B$ is hyperbolic if and only if $\sigma(B)\cap
i\R=\emptyset,$ see details in \cite[Proposition 1.15, pp.305]{EN}.

\begin{definition}
Let $\alpha\in(0,1)$. A Banach space $(\X_\alpha,
\|\cdot\|_\alpha)$ is said to be an intermediate
 space between $D(B)$ and $\X$, or a space of class ${\mathcal J}_\alpha$,
 if $D(B)\subset \X_\alpha\subset \X$ and there is a constant  $c>0$
 such that
 \begin{equation}\label{extrap}
 \|x\|_{\alpha}\leq
c\|x\|^{1-\alpha}\|x\|_{B}^{\alpha},\qquad x\in D(B).
\end{equation}
\end{definition}

Concrete examples of $\X_\alpha$ include $D((-B^\alpha))$ for
$\alpha\in(0,1)$, the domains
 of the fractional powers of $B$,  the real interpolation
 spaces $D_B(\alpha,\infty)$,  $\alpha\in(0,1)$,  defined as
 the space of all $x \in \X$ such that,
 $$[x]_\alpha=\sup_{0< t\leq1}\|t^{1-\alpha}BT(t)x\|<\infty$$
with the norm
$$\|x\|_{\alpha}=\|x\|+[x]_{\alpha},$$ the abstract H\"older spaces
$D_B(\alpha):=\overline{D(B)}^{\|.\|_\alpha}$ as well as the
complex interpolation spaces $[\X,D(B)]_\alpha$.

For a hyperbolic analytic  semigroup $(T(t))_{t\geq0}$, one can easily check that
  similar estimations as both Eq. \eqref{hyP} and Eq. \eqref{hyQ} still hold with the $\alpha$-norms
  $\|\cdot\|_\alpha$. In fact,  as the part of $A$ in $R(Q)$ is
  bounded, it follows from Eq. \eqref{hyQ} that
  \begin{equation*}
  \|BT(t)Qx\|\leq
C'e^{\delta t}\|x\| \ \ \mbox{for} \ \ t\leq0. \end{equation*}

 Hence, from Eq. \eqref{extrap} there exists a constant $c(\alpha)>0$  such that
\begin{equation}\label{hyQ*}
\|T(t)Qx\|_{\alpha}\leq c(\alpha)e^{\delta t}\|x\|\ \ \mbox{ for }
t\leq0.
\end{equation}

In addition to the above, the following holds
  \begin{equation*}
\|T(t)Px\|_{\alpha}\leq \|T(1)\|_{B(\X,\X_\alpha)}\|T(t-1)Px\|,\ \
t\geq1,
\end{equation*}
 and hence from Eq. \eqref{hyP}, one obtains
\begin{equation*}
\|T(t)Px\|_{\alpha}\leq M'e^{-\delta t}\|x\|,  \qquad
t\geq1,\end{equation*} where $M'$ depends on $\alpha$. For
$t\in(0,1]$, by Eq. \eqref{analy} and Eq. \eqref{extrap},
$$\|T(t)Px\|_{\alpha}\leq M'' t^{-\alpha}\|x\|.$$
Hence, there exist constants $M(\alpha)>0$ and $\gamma>0$ such
that
  \begin{equation}\label{hyP*}
\|T(t)Px\|_{\alpha}\leq M(\alpha)t^{-\alpha}e^{-\gamma t}\|x\|
\qquad \mbox{ for } t>0.
\end{equation}

\begin{remark} 
Note that if the analytic semigroup $T(t)$ is exponential stable, that is, there exists constants $N, \delta > 0$ such that $\|T(t)\| \leq N e^{-\delta t}$ for $t \geq 0$, then the projection $P = I$ ($Q = I - P(t) = 0$). In that case, Eq. (\ref{hyP*}) still holds and can be rewritten as
follows: for all $x \in \X$,
\begin{equation}\label{eq1.100}
  \left\|T(t)x\right\|_{\alpha}\leq
 M(\alpha)e^{- \frac{\gamma}{2}t} t^{-\alpha} \left\|x\right\|.
  \end{equation}
\end{remark}

For more on interpolation spaces and related issues, we refer the reader to the following excellent books
Amann \cite{Am} and Lunardi \cite{Lun}.

\subsection{Pseudo-Almost Automorphic Functions}
Let $BC(\R, \X)$ stand for the Banach space of all bounded continuous functions $\ph: \R \mapsto \X$, which we equip with the sup-norm defined by
$\|\ph\|_\infty: = \sup_{t \in \R} \left\|\ph(t)\right\|$
for all $\ph \in BC(\R, \X)$.
If $\beta \geq 0$, 
we will also be using the following notions,
$$\|\Phi\|_{\E_\beta, \infty} := \sup_{t \in \R} \|\Phi(t)\|_{\E_\beta}$$ for $\Phi \in BC(\R, \E_\beta)$, and $$\|\ph\|_{\beta, \infty} := \sup_{t \in \R} \|\ph(t)\|_{\beta}$$ for $\ph \in BC(\R, D(A^\beta))$.

\begin{definition}\cite{TDbook} \label{DDD}
A function $f\in C(\R,\X)$ is said to be almost automorphic if for
every sequence of real numbers $(s'_n)_{n \in \N}$, there
   exists a subsequence $(s_n)_{n \in \N}$ such that
      $$ g(t):=\lim_{n\to\infty}f(t+s_n)$$
   is well defined for each $t\in\mathbb{R}$, and
      $$ \lim_{n\to\infty}g(t-s_n)=f(t)$$
   for each $t\in \mathbb{R}$.
\end{definition}

If the convergence above is uniform in $t\in \R$, then $f$ is
almost periodic in the classical Bochner's sense. Denote by
$AA(\X)$ the collection of all almost automorphic functions
$\R\mapsto \X$. Note that $AA(\X)$ equipped with the sup-norm
turns out to be a Banach space.

Among other things, almost automorphic functions satisfy the
following properties.

\begin{theorem}\cite{TDbook}\label{T}
   If $f, f_1, f_2\in AA(\X)$, then
   \begin{itemize}
      \item[(i)] $f_1+f_2\in AA(\X)$,
      \item[(ii)] $\lambda f\in AA(\X)$ for any scalar $\lambda$,
      \item[(iii)] $f_\alpha\in AA(\X)$ where $f_\alpha:\mathbb{R}\to \X$ is defined by
                     $f_\alpha(\cdot)=f(\cdot+\alpha)$,
      \item[(iv)] the range $\mathcal{R}_f:=\big\{f(t):t\in\mathbb{R}\big\}$ is relatively
                    compact in $\X$, thus $f$ is bounded in norm,
      \item[(v)] if $f_n\to f$ uniformly on $\mathbb{R}$ where each $f_n\in AA(\X)$, then $f\in
                   AA(\X)$ too.
   \end{itemize}
\end{theorem}

\begin{definition}\label{KKK}Let $\Y$ be another Banach space.
A jointly continuous function $F: \R \times \Y \mapsto \X$ is said
to be almost automorphic in $t \in \R$ if $t \mapsto F(t,x)$ is
almost automorphic for all $x \in K$ ($K \subset \Y$ being any
bounded subset). Equivalently, for every sequence of real numbers
$(s'_n)_{n \in \N}$, there
   exists a subsequence $(s_n)_{n \in \N}$ such that
      $$G(t, x):=\lim_{n\to\infty}F(t+s_n, x)$$
   is well defined in $t\in\mathbb{R}$ and for each $x \in K$, and
      $$ \lim_{n\to\infty}G(t-s_n, x)=F(t, x)$$
   for all $t\in \mathbb{R}$ and $x \in K$.

 The collection of such functions will be denoted by $AA(\R \times \X)$.
\end{definition}

 For
more on almost automorphic functions and their generalizations, we refer
the reader to the recent book by Diagana \cite{TDbook}.

Define (see Diagana \cite{TDbook, TT}) the space 
$PAP_0(\R, \X)$ as the collection of all functions $\ph \in BC(\R, \X)$ satisfying, $$\lim_{r \to \infty}
\displaystyle{\frac{1}{2r}} \int_{-r}^r \|\ph(s)\| ds =
0.$$

Similarly, $PAP_0(\R \times \X)$ will denote the collection of all
bounded continuous functions $F: \R \times \Y \mapsto \X$ such
that
$$\lim_{T \to \infty}
\displaystyle{\frac{1}{2r}} \int_{-r}^r \| F(s, x)\| ds =
0$$ uniformly in $x \in K$, where $K \subset \Y$ is any bounded
subset.

\begin{definition}\label{DEF} (Liang {\it et al.} \cite{L} and {\it Xiao et al.} \cite{LL}) 
A function $f \in BC(\R, \X)$ is called pseudo almost automorphic
if it can be expressed as $f = g + \phi,$ where $g \in AA(\X)$ and
$\phi \in PAP_0(\X)$. The collection of such functions will be
denoted by $PAA({\mathbb X})$.
\end{definition}

The functions $g$ and $\phi$ appearing in Definition~\ref{DEF} are
respectively called the {\it almost automorphic} and the {\it
ergodic perturbation} components of $f$.

\begin{definition} Let $\Y$ be another Banach space.
A bounded continuous function $F: \R \times \Y \mapsto \X$ belongs
to $AA(\R \times \X)$ whenever it can be expressed as $F = G + \Phi,$
where $G\in AA(\R \times \X)$ and $\Phi \in PAP_0(\R \times \X)$. The
collection of such functions will be denoted by $PAA(\R \times \X)$.
\end{definition}

A substantial result is the next theorem, which is due to Xiao et
al. \cite{LL}.

\begin{theorem}\label{MN} \cite{LL} The space $PAA(\X)$ equipped with the sup
norm $\|\cdot\|_\infty$ is a Banach space.
\end{theorem}

\begin{theorem}\cite{LL}\label{CO} If $\Y$ is another Banach space, $f: \R \times \Y \mapsto \X$ belongs to $PAA(\R \times \X)$
and if $x \mapsto f(t,x)$ is uniformly continuous on each bounded
subset $K$ of $\Y$ uniformly in $t \in \R$, then the function defined
by $h(t) = f(t, \varphi(t))$ belongs to $PAA(\X)$ provided
$\varphi \in PAA(\Y)$.
\end{theorem}

 For
more on pseudo-almost automorphic functions and their generalizations, we refer
the reader to the recent book by Diagana \cite{TDbook}.

\section{Main Results}Fix $\beta \in (0, 1)$.
Consider the first-order differential equations,
\begin{eqnarray}\label{LT}
\frac{d\ph}{dt} = A \ph + g(t), \ \ t \in \R,
\end{eqnarray}and
\begin{eqnarray}\label{LT2}
\frac{d\ph}{dt} = (A +B)\ph + f(t, \ph), \ \ t \in \R,
\end{eqnarray}
where $A: D(A) \subset \X \mapsto \X$ is a sectorial linear operator on a Banach space $\X$, $B: C(\R, D(A)) \mapsto \X$ is a linear operator, and $g: \R \mapsto \X$ and $f: \R \times \X \mapsto \X$ are bounded continuous functions.

To study the existence of pseudo-almost automorphic mild
solutions to Eq. (\ref{LT}) (and hence Eq. (\ref{LT2})), we will need the following assumptions,

\begin{enumerate}
  \item [(H.1)] The linear operator $A$ is sectorial. Moreover, if $T(t)$ denotes the analytic semigroup associated with it, we suppose that $T(t)$
is hyperbolic, that is, $$\sigma(A) \cap i\R = \emptyset.$$

\item[(H.2)] The semigroup $T(t)$ is not only compact for $t > 0$ but also is exponentially stable, i.e., there exists constants $N, \delta > 0$ such that $$\|T(t)\| \leq N e^{-\delta t}$$ for $t \geq 0$.

\item[(H.3)] The linear operator $B: BC(\R, \X_{\beta}) \mapsto \X$, where $\X_{\beta}:= D((-A)^{\beta})$, is bounded . Moreover, the following holds,
$$\displaystyle C_0 := \|B\|_{B(BC(\R, \X_{\beta}), \X)} \leq \frac{1}{2d(\beta)},$$ where $d (\beta) := M(\beta) (2\delta^{-1})^{1-\beta} \Gamma(1-\beta)$.
\item[(H.4)] 
The function $f: \R \times \X_\beta \mapsto \X$ is pseudo-almost automorphic in the first variable uniformly in the second one. For each bounded subset $K \subset \X_\beta$, $f(\R, K)$ is bounded. Moreover, the function $u \mapsto f(t,u)$ is uniformly continuous on any bounded
subset $K$ of $\X_\beta$ for each $t \in \R$. Finally, we suppose that there exists $L > 0$ such that
$$\sup_{t \in \R, \ \ \|\ph\|_{\beta} \leq L} \Big\|f(t,\ph)\Big\| \leq \frac{L}{ 2d(\beta)}.$$
\item[(H.5)] If $(u_n)_{n \in \N} \subset PAA(\X_\beta)$ is uniformly bounded and uniformly convergent upon every compact subset of $\R$, then $f(\cdot, u_n(\cdot))$ is relatively compact in $BC(\R, \X)$.
\end{enumerate}

\begin{remark}\label{map} Note that if (H.3) holds, then it can be easily shown that the linear operator $B$ maps $PAA(\X_\beta)$ into $PAA(\X)$.
\end{remark}

\begin{definition}
Under assumption (H.1), a continuous function $\ph: \R \mapsto \X$ is said to be a mild
solution to Eq. (\ref{LT}) provided that
\begin{eqnarray}\label{M}
\ph(t)=T(t-s) \ph(s) + \int_{s}^{t}T(t-\tau) g(\tau)d\tau, \quad \forall (t,s) \in \T.
\end{eqnarray}
\end{definition}

\begin{lemma}\cite{TDbook}\label{R}
Suppose assumptions {\rm (H.1)--(H.2)} hold. If 
$g: \R \mapsto \X$ is a bounded continuous function, then $\ph$ given by
\begin{eqnarray}\label{VCF1} \ph(t) := \int_{-\infty}^t T(t-s) g(s) ds\end{eqnarray}
for all $t\in \R$, is the unique bounded mild solution to Eq. (\ref{LT}).
\end{lemma}

\begin{definition}
Under assumptions (H.1), (H.2), and (H.3) and if 
$f: \R \times \X_\beta \mapsto \X$ is a bounded continuous function, then a continuous function $\ph: \R \mapsto \X_\beta$ satisfying 
\begin{eqnarray}\label{VCF} \ph(t) = T(t-s) \ph(s) + \int_{s}^{t} T(t-s) \Big[B \ph(s) + f(s, \ph(s))\Big] ds,\quad \forall (t,s) \in \T\end{eqnarray}
is called a mild solution to Eq. (\ref{LT2}).
\end{definition}

Under assumptions (H.1), (H.2), and (H.3) and if 
$f: \R \times \X_\beta \mapsto \X$ is a bounded continuous function, it can be shown that the function $\ph: \R \mapsto \X_\beta$ defined by 
\begin{eqnarray}\label{VCF} \ph(t) = \int_{-\infty}^{t} T(t-s) \Big[B\ph(s) + f(s, \ph(s))\Big] ds\end{eqnarray}
for all $t\in \R$, is a mild solution to Eq. (\ref{LT2}).

Define the following integral operator,
$$(S \ph)(t) =\int_{-\infty}^{t}T(t-s) \Big[B \ph(s) + f(s, \ph(s))\Big] ds.$$

We have

\begin{lemma}\label{B}Under assumptions {\rm (H.1)--(H.2)--(H.3)} and if 
$f: \R \times \X_\beta \mapsto \X$ is a bounded continuous function, then the mapping $S: BC(\R,\X_\beta) \mapsto BC(\R, \X_\beta)$ is well-defined and continuous.

\end{lemma}

\begin{proof}
We first show that $S$ is well-defined and that $S (BC(\R, \X_\beta)) \subset BC(\R, \X_\beta)$. Indeed, letting $u \in BC(\R, \X_\beta)$, $g(t) := f(t,u(t))$, and using Eq. (\ref{eq1.100}),
we obtain
\begin{eqnarray*}
\big\|S u(t)\big\|_\beta &\leq & \int_{-\infty }^t\big\|T(t-s) [B u(s) + g(s)]
\big\|_\beta ds\nonumber\\
& \leq & \int_{-\infty }^t
M(\beta)e^{- \frac{\delta}{2}(t-s)}(t-s)^{1-\beta} 
\Big[\|B u(s)\| + \|g(s)\|\Big]ds\nonumber\\
& \leq & \int_{-\infty }^t
M(\beta)e^{- \frac{\delta}{2}(t-s)}(t-s)^{1-\beta} 
\Big[C_0\|u(s)\|_\beta + \|g(s)\|\Big]ds\nonumber\\
& \leq & d(\beta) \Big(C_0 \|u\|_{\beta, \infty} + \|g\|_\infty\Big)\label{bound},
\end{eqnarray*} 
for all $t \in \R$, where $d :=M(\beta) (2\delta^{-1})^{1-\beta} \Gamma(1-\beta)$, and hence $Su: \R \mapsto \X_\beta$ is bounded.

To complete the proof it remains to show that $S$ is continuous. For that, set$$F(s, u(s)) := B u(s) + g(s) = B u(s) + f(s, u(s)), \ \ \forall s \in \R.$$ 

Consider an arbitrary sequence of
functions
$u_n \in BC(\R, \X_\beta)$ that converges uniformly to some $u \in BC(\R, \X_\beta)$, that is,
$\big\|u_n -u\big\|_{\beta, \infty} \to 0 \quad \mbox{as} \ \ n \to \infty.$

Now 
\begin{eqnarray*}
\big\|Su(t) - Su_n(t)\big\|_\beta &=& \big\|\int_{-\infty}^t T(t-s) [F(s, u_n(s))-F(s, u(s))]\>ds\big\|_{\beta}\\
&\leq& M(\beta) \int_{-\infty}^t (t-s)^{-\beta} e^{-\frac{\delta}{2}\>(t-s)}
\big\|F(s, u_n(s))-F(s, u(s))\big\|\>ds.\\
&\leq& M(\beta) \int_{-\infty}^t (t-s)^{-\beta} e^{-\frac{\delta}{2}\>(t-s)}
\big\|f(s, u_n(s))-f(s, u(s))\big\|\>ds\\
&+& M(\beta) \int_{-\infty}^t (t-s)^{-\beta} e^{-\frac{\delta}{2}\>(t-s)}
\big\|B (u_n(s)- u(s))\big\|\>ds\\
&\leq& M(\beta) \int_{-\infty}^t (t-s)^{-\beta} e^{-\frac{\delta}{2}\>(t-s)}
\big\|f(s, u_n(s))-f(s, u(s))\big\|\>ds\\
&+& d(\beta) C_0 \,\|u_n - u\|_{\beta, \infty}.\\
\end{eqnarray*}

Using the continuity of the function $f: \R \times \X_\beta \mapsto \X$ and the Lebesgue Dominated Convergence Theorem we conclude that
\begin{eqnarray*}
\big\|\int_{-\infty}^t T(t-s) P(s) [f(s, u_n(s))-f(s, u(s))]\>ds\big\| \to 0\>\>\>\mbox{as}\>\>\> n\to\infty.
\end{eqnarray*}
Therefore,
$\big\|S u_n - S u\big\|_{\beta, \infty} \to 0$
as $n \to \infty$. The proof is complete.

\end{proof}

\begin{lemma}\label{POL} Under assumptions {\rm (H.1)---(H.4)}, then $S(PAA(\X_\beta) \subset PAA(\X_\beta)$.

\end{lemma}

\begin{proof} Let $u \in PAA(\X_\beta)$ and define $h(s): = f(s, u(s)) + B u(s)$ for all $s \in \R$. Using (H.4) and Theorem \ref{CO} it follows that the function $s \mapsto f(s, u(s))$ belongs to $PAA(\X)$.
Similarly, using Remark \ref{map} it follows that the function $s \mapsto B u(s)$ belongs to $PAA(\X)$.
In view of the above, the function $s \mapsto h(s)$ belongs to $PAA(\X)$. 

Now write $h= h_1 + h_2 \in PAA (\X)$ where $h_1 \in AA (\X)$ and $h_2 \in PAP_0(\X)$ and set 
$$R h_j(t) :=  \int_{-\infty }^t
T(t-s) h_j(s)ds  \ \ \mbox{for all} \ \ t \in \R, \ \ j=1, 2.$$

Our first task consists of showing that $R \big(AA(\X)\big) \subset AA(\X_\beta)$. Indeed,
using the fact that $h_1 \in AA (\X)$, for every
sequence of real numbers $(\tau'_n)_{n \in \N}$ there
   exist a subsequence $(\tau_n)_{n \in \N}$ and a function $f_1$ such that
      $$f_1 (t):=\lim_{n\to\infty} h_1 (t+\tau_n)$$
   is well defined for each $t\in\mathbb{R}$, and
      $$\lim_{n\to\infty} f_1 (t-\tau_n)= h_1 (t)$$
   for each $t\in \mathbb{R}$.

Now
\begin{eqnarray*}
(R h_1)(t + \tau_n) - (R f_1)(t) &=& \int_{-\infty}^{t+\tau_n} T(t
+\tau_n -s)
h_1(s)ds -  \int_{-\infty}^{t} T(t-s) f_1(s) ds\nonumber\\
&=& \int_{-\infty}^{t} T(t-s)
h_1(s+\tau_n)ds -\int_{-\infty}^{t} T(t-s)
f_1(s)ds.\nonumber\\
&=& \int_{-\infty}^{t} T(t-s)
\Big(h_1(s+\tau_n) - f_1(s)\Big)ds\nonumber.\\
\end{eqnarray*}

From Eq. (\ref{eq1.100}) and the Lebesgue Dominated Convergence Theorem, it easily follows that
\begin{eqnarray*}
\Big\|\int_{-\infty}^{t} T(t-s) \Big(h_1(s+\tau_n) - f_1(s)\Big)ds \Big\|_\beta &\leq& \int_{-\infty}^{t} \Big\|T(t-s) \Big(h_1(s+\tau_n) - f_1(s)\Big)\Big\|_\beta ds \nonumber\\
&\leq&M(\beta) \int_{-\infty}^{t} (t-s)^{-\beta} e^{-\frac{\delta}{2}(t-s)}
\|h_1(s+\tau_n) - f_1(s)\| ds\\
&\to& 0 \ \ \mbox{as} \ \ n \to \infty,
\end{eqnarray*}
and hence
$$(R f_1) (t) = \lim_{n \to \infty} (R h_1) (t+\tau_n)$$ for all $t \in
\R$.

Using similar arguments as above one obtains that 
$$(R h_1) (t) = \lim_{n \to \infty} (R f_1) (t-\tau_n)$$ for all $t \in
\R$, which yields, $t \mapsto (S h_1)(t)$ belongs to $AA(\X_\beta)$.

The next step consists of showing that $R \big(PAP_0 (\X)\big) \subset PAP_0(\X_\beta)$. Obviously, $R h_2 \in BC (\R, \X_\beta)$ (see Lemma \ref{B}).
Using the fact that  $h_2 \in PAP_0(\X)$ and Eq. (\ref{eq1.100}) it can be easily shown that $(R h_2) \in PAP_0 (\X_\beta)$. Indeed, for $r > 0$,
\begin{eqnarray} \displaystyle \frac{1}{2r} \int_{-r}^r \Big\|\int_{-\infty}^{t} T(t-s) h_2(s)ds\Big\|_\beta  dt 
&\leq& \frac{M(\beta)}{2r}\int_{-r}^{r}\int_0^{\infty}
e^{\frac{\delta}{2} s} s^{-\beta} \Big\|h_2(t-s)\Big\|
dsdt\nonumber\\
&\leq& 
M(\beta)\int_0^{\infty} e^{\frac{\delta}{2} s} s^{-\beta} \left(\frac{1}{2r}\int_{-r}^{r} \Big\|h_2(t-s)\Big\|
dt\right) ds.\nonumber
\end{eqnarray}

Using the fact that $PAP_0(\X)$ is translation-invariant it follows that
$$\displaystyle \lim_{r \to \infty} \frac{1}{2r}\int_{-r}^{r}
\Big\|h_2(t-s)\Big\| dt  = 0,$$ as $t \mapsto h_2(t-s) \in
PAP_0(\X)$ for every $s\in \R$. 

One completes the proof by
using the Lebesgue Dominated Convergence Theorem.
In summary, $(R h_2) \in  PAP_0 (\X_\beta)$, which completes the proof.

\end{proof}

\begin{theorem}\label{AB}
Suppose assumptions {\rm (H.1)---(H.5)} hold, then Eq. (\ref{LT2}) has at least one pseudo-almost automorphic mild solution
\end{theorem}

\begin{proof}  Let $B_\beta = \{u \in PAA(\X_\beta): \|u\|_\beta \leq L\}$. Using the proof of Lemma \ref{B} it follows that $B_\beta$ is a convex and closed set. Now using Lemma \ref{POL} it follows that $S(B_\beta) \subset PAA(\X_\beta)$. 

Now for all $u \in B_\beta$,
\begin{eqnarray*}
\big\|S u(t)\big\|_\beta &\leq & \int_{-\infty }^t\big\|T(t-s) [B u(s) + g(s)]
\big\|_\beta ds\nonumber\\
& \leq & \int_{-\infty }^t
M(\beta)e^{- \frac{\delta}{2}(t-s)}(t-s)^{-\beta} 
\Big[\|B u(s)\| + \|f(s, u(s))\|\Big]ds\nonumber\\
& \leq & \int_{-\infty }^t
M(\beta)e^{- \frac{\delta}{2}(t-s)}(t-s)^{-\beta} 
\Big[C_0\|u(s)\|_\beta + \|f(s, u(s))\|\Big]ds\nonumber\\
& \leq & d(\beta) \Big(\frac{ L}{2d(\beta)} + \frac{L}{2d(\beta)}\Big)\\
&=& L
\end{eqnarray*} 
for all $t \in \R$, and hence $Su \in B_\beta$. 

To complete the proof, we have to prove the following:
\begin{enumerate}
\item[a)] That $V = \{ Su(t): u \in B_\beta\}$ is a relatively compact subset of $\X_\beta$ for each $t \in \R$;
\item[b)] That $W = \{ Su: u \in B_\beta\} \subset BC(\R, \X_\beta)$ is equi-continuous.
\end{enumerate}

To show a), fix $t \in \R$ and consider an arbitrary $\varepsilon > 0$.

Now
\begin{eqnarray*}
(S_{\varepsilon} u)(t) &:=& \int_{-\infty}^{t-\varepsilon}T(t-s) F(s, u(s))ds, \ u \in B_\beta\\
&=& T(\varepsilon) \int_{-\infty}^{t-\varepsilon} T(t-\varepsilon-s) F(s, u(s)) ds, \ u \in B_\beta\\
&=& T(\varepsilon) (S u)(t-\varepsilon), \ u \in B_\beta
\end{eqnarray*}
and hence $V_\varepsilon := \{ S_{\varepsilon} u(t): u \in B_\beta\}$ is relatively compact in $\X_\beta$ as the evolution family $T(\varepsilon)$ is compact by assumption.

Now

\begin{eqnarray*}
&&\big\|S u(t) - T(\varepsilon) \int_{-\infty}^{t-\varepsilon} T(t-\varepsilon-s) F(s, u(s)) ds\big\|_\beta \\
&&\leq  \int_{t-\varepsilon}^t \|T(t-s) F(s, u(s)) \|_\beta ds\\
&&\leq  M(\beta)\int_{t-\varepsilon}^t e^{- \frac{\delta}{2}(t-s)}(t-s)^{-\beta} \left\|F(s, u(s))\right\| ds\\
&&\leq M(\beta) \int_{t-\varepsilon}^t (t-s)^{-\beta} e^{-\frac{\delta}{2}\>(t-s)}
\big\|g(s)\big\|\>ds + M(\beta) \int_{t-\varepsilon}^t (t-s)^{-\beta} e^{-\frac{\delta}{2}\>(t-s)}
\big\|B u(s)\big\|\>ds\\
&&\leq M(\beta) \int_{t-\varepsilon}^t (t-s)^{-\beta} e^{-\frac{\delta}{2}\>(t-s)}
\big\|f(s, u(s))\big\|\>ds
+ M(\beta) C_0 \|u\|_{\beta, \infty} \int_{t-\varepsilon}^t (t-s)^{-\beta} e^{-\frac{\delta}{2}\>(t-s)}\>ds\\
&&\leq M(\beta) L\Big(d^{-1}(\beta) + C_0\Big) \int_{t-\varepsilon}^t (t-s)^{-\alpha} \>ds\\
&&= M(\beta) L\Big(d^{-1}(\beta) + C_0\Big)\varepsilon^{1-\beta} (1-\beta)^{-1},
\end{eqnarray*}
and hence the set $V := \{Su(t): u \in B_\beta\} \subset \X_\beta$ is relatively compact.

The proof for b) follows along the same lines as in Li {\it et al.} \cite[Theorem 31]{li} and hence is omitted. 

The rest of the proof slightly follows along the same lines as in Diagana \cite{MCM}. Indeed, since $B_\beta$ is a closed convex subset of $PAA(\X_\beta)$ and that $S(B_\beta) \subset B_\beta$, it follows that $\overline{co}\,{S(B_\beta)} \subset B_\beta.$ Consequently, 
$$S(\overline{co}\, S(B_\beta)) \subset S(B_\beta) \subset \overline{co}\ S(B_\beta).$$ Further, it is not hard to see that $\{u(t): u \in \overline{co}\,{S(B_\beta)}\}$ is relatively compact in $\X_\beta$
for each fixed $t \in \R$ and that functions in $\overline{co}\,{S(B_\beta)}$ are equi-continuous on $\R$. 
Using Arzel\`a-Ascoli theorem, we deduce that the restriction of $\overline{co}\,{S(B_\beta)}$ to any
compact subset $I$ of $\R$ is relatively compact in $C(I, \X_\beta)$. 

In summary, $S: \overline{co}\,{S(B_\beta)} \mapsto \overline{co}\,{S(B_\beta)}$ is continuous and compact. 
Using the Schauder fixed point it follows that $S$ has a fixed-point, which obviously is a pseudo-almost automorphic
mild solution to Eq. (\ref{LT2}).

\end{proof}

Fix $\alpha \in [\frac{1}{2}, 1)$. 
In order to study Eq. (\ref{OPR1}), we let $\beta = \frac{1}{2}$ and suppose that the following additional assumption holds:
\begin{enumerate}
\item[(H.6)] There exists a function $\rho \in L^1(\R, (0, \infty))$ with $\displaystyle \|\rho\|_{L^1(\R, (0, \infty))} \leq \frac{1}{2d(\frac{1}{2})}$ such that
$$\Big\|\c(t) \ph\Big\|_{\E} \leq \rho(t) \big\|\ph\big\|_{\E_{\frac{1}{2}}}$$ for all $\ph \in \E_{\frac{1}{2}}$ and $t \in \R$.
\end{enumerate}

\begin{corollary}\label{ABB} 
Under assumptions {\rm (H.1)--(H.2)--(H.4)--(H.5)--(H.6)}, then Eq. (\ref{OPR1}) (and hence Eq. (\ref{PRR1}) and Eq. (\ref{PRRR})) has at least one pseudo-almost automorphic mild solution.
\end{corollary}

\begin{proof} It suffices to show that $\a$ and $\b$ satisfy similar assumptions as {\rm (H.1)--(H.2)--(H.3)} and that 
$F$ satisfies similar assumptions as {\rm (H.4)--(H.5)}.

{\bf Step 1}. Assumption (H.6) yields $\b$ satisfies similar assumption as (H.3), where $\b$ is defined by
$$\b \ph(t)  := \int_{-\infty}^t \c(t-s) \ph (s) ds.$$

Indeed, since  the function $\rho$ is integrable, it is clear that the operator $\b$ belong to $B(BC(\R, \E_{\frac{1}{2}}), \E)$ with $\|\b\|_{B(BC(\R, \E_{\frac{1}{2}}), \E)} \leq \|\rho\|_{L^1(\R, (0,\infty))}$. In fact, we take $C_0 = \|\rho\|_{L^1(\R, (0,\infty))}$. The fact the function $t \mapsto \b\ph (t)$ is pseudo-almost automorphic for any $\ph \in PAA(\E_{\frac{1}{2}})$ is guaranteed by Remark \ref{map}. However, for the sake of clarity, we will show it. Indeed, write $\ph = \ph_1 + \ph_2$, where $\ph_1 \in AA(\E_{\frac{1}{2}})$ and $\ph_2 \in PAP_0(\E_{\frac{1}{2}})$. 
Using the fact that the function $t \mapsto \ph_1(t)$ belongs to $AA (\E_{\frac{1}{2}})$, for every
sequence of real numbers $(\tau'_n)_{n \in \N}$ there
   exist a subsequence $(\tau_n)_{n \in \N}$ and a function $\psi_1$ such that
      $$\psi_1 (t):=\lim_{n\to\infty} \ph_1(t+\tau_n)$$
   is well defined for each $t\in\mathbb{R}$, and
      $$\lim_{n\to\infty} \psi_1 (t-\tau_n)= \ph_1 (t)$$
   for each $t\in \mathbb{R}$.

Now
\begin{eqnarray*}
\b \ph_1 (t + \tau_n) - \b \psi_1 (t) &=& \int_{-\infty}^{t+\tau_n} \c(t
+\tau_n - s)
\ph_1 (s)ds -  \int_{-\infty}^{t} \c(t-s) \psi_1(s) ds\nonumber\\
&=& \int_{-\infty}^{t} \c(t -s)
\ph_1(s+\tau_n)ds -\int_{-\infty}^{t} C(t-s)
\psi_1(s)ds\nonumber\\
&=& \int_{-\infty}^{t} \c(t-s)
\Big(\ph_1(s+\tau_n) - \psi_1(s)\Big)ds\nonumber\\
\end{eqnarray*}
and hence
\begin{eqnarray*}\Big\|\b \ph_1 (t + \tau_n) - \b \psi_1 (t)\Big\|_{\E} &\leq& \Big\|\int_{-\infty}^{t} \c(t-s)
\Big(\ph_1(s+\tau_n) - \psi_1(s)\Big)ds\Big\|_{\E}\\
&\leq& \int_{-\infty}^{t} \Big\|\c(t-s)
\Big(\ph_1(s+\tau_n) - \psi_1(s)\Big)\Big\|_{\E} ds\\
&\leq& \int_{-\infty}^{t} \rho(t-s)
\Big\|\ph_1(s+\tau_n) - \psi_1(s)\Big\|_{\E_{\frac{1}{2}}} ds\\
\end{eqnarray*}
which by Lebesgue Dominated Convergence Theorem yields
$$\lim_{n \to \infty} \Big\|\b \ph_1 (t + \tau_n) - \b \psi_1 (t)\Big\|_{\E} = 0.$$

Using similar arguments, we obtain
$$\lim_{n \to \infty} \Big\|\b\psi_1(t - \tau_n) - \b\phi_1(t)\Big\|_{\E} = 0.$$

For $r > 0$,
\begin{eqnarray} \displaystyle \frac{1}{2r} \int_{-r}^r \Big\|\int_{-\infty}^{t} \c(t-s) \ph_2(s)ds\Big\|_{\E}  dt 
&\leq& \frac{1}{2r}\int_{-r}^{r}\int_0^{\infty}
\rho(s) \Big\|\ph_2(t-s)\Big\|_{\E_{\frac{1}{2}}}
dsdt\nonumber\\
&\leq& 
\int_0^{\infty} \rho(s) \left(\frac{1}{2r}\int_{-r}^{r} \Big\|\ph_2(t-s)\Big\|_{\E_{\frac{1}{2}}}
dt\right) ds.\nonumber
\end{eqnarray}

Now
$$\displaystyle \lim_{r \to \infty} \frac{1}{2r}\int_{-r}^{r}
\Big\|\ph_2(t-s)\Big\|_{\E_{\frac{1}{2}}} dt  = 0,$$ as $t \mapsto \ph_2(t-s) \in
PAP_0(\E_{\frac{1}{2}})$ for every $s\in \R$. 

Therefore,
$$\lim_{r \to \infty} \frac{1}{2r} \int_{-r}^r \Big\|\int_{-\infty}^{t} \c(t-s) \ph_2(s)ds\Big\|_{\E}  dt = 0$$
by
using the Lebesgue Dominated Convergence Theorem.
\\
{\bf Step 2}. Clearly, the operator $\a$ satisfies similar assumptions as {\rm (H.1)--(H.2)} in the space $\E_{\frac{1}{2}}$. Indeed, for all $\ph \in D(\a)$,
we have
\begin{align*}
\a\ph = \sum_{n=1}^{\infty} \a_n  P_n \ph,
\end{align*}
where \begin{equation*} P_n := \left (\begin{matrix}
           E_n & 0\\ \\
           0 & E_n\\
         \end{matrix}
       \right) \ \ \mbox{and} \ \
\a_n := \left (\begin{matrix}
           0 & 1\\ \\
           -\lambda_n & -\lambda_n^\alpha \\
         \end{matrix}
       \right),\; n \geq 1.
      \end{equation*}

The characteristic equation for $\a_n$ is given by
\begin{equation*}\label{D}
\rho^2 + 2\gamma \lambda_n^\alpha \rho + \lambda_n= 0,
\end{equation*}
from which we obtain its eigenvalues given by
$$\rho_1^n =  \lambda_n^\alpha \Big(-\gamma + \sqrt{\gamma^2 - \lambda_n^{1-2\alpha}}\Big) \ \ \mbox{and} \ \ \lambda_2^n = \lambda_n^\alpha \Big(-\gamma - \sqrt{\gamma^2 - \lambda_n^{1-2\alpha}}\Big),$$and hence
 $\sigma(\a_n) = \Big\{\rho_1^n , \rho_2^n \Big\}.$

Using Eq. (\ref{AS1}) it follows that there exists $\omega > 0 $ such that 
$\rho(\a)$ contains the halfplane $$S_{\omega} :=\Big\{\lambda \in \C: \Re e\, \lambda \geq \omega \Big\}.$$

Now since $\rho_1^n$ and $\rho_2^n$ are distinct and that each of them is of multiplicity one, then $\a_n$ is diagonalizable. Further, it is not difficult to see that
$\a_n = K_n^{-1} J_n K_n$, where $J_n, K_n$ and $K_{n}^{-1}$ are respectively given by

\begin{equation*}
J_n = \left (
\begin{matrix}
\rho_1^n & 0\\ \\
0 & \rho_{2}^n\\
\end{matrix}
\right), \ \ \ \displaystyle K_n = \left (
\begin{matrix}
1 & 1\\ \\
\rho_{1}^n&\rho_2^n\\
\end{matrix}
\right),
\end{equation*}
and

\begin{equation*}
\displaystyle K_n^{-1}= \frac{1}{\rho_1^n - \rho_2^n}\left (
\begin{matrix}
-\rho_2^n & 1\\ \\
\rho_{1}^n&-1\\
\end{matrix}
\right).
\end{equation*}

For $\lambda \in S_{\omega}$ and $\ph\in \E_{\frac{1}{2}}$, one has
\begin{align*}
R(\lambda, \a)\ph &=\displaystyle
\sum_{n=1}^{\infty}(\lambda-\a_n)^{-1}P_n \ph\\
&= \sum_{n=1}^{\infty}K_n(\lambda-
J_n )^{-1}K_n^{-1}P_n\ph.
\end{align*}
Hence,
\begin{align*}
\Big\|R(\lambda, \a)\ph\Big\|_{\E_{\frac{1}{2}}}^2 &\leq  \displaystyle
\sum_{n=1}^{\infty}\Big\|K_n(\lambda-J_n)^{-1}K_n^{-1}\Big\|^2
\Big\|P_n\ph\Big\|_{\E_{\frac{1}{2}}}^2\\
&\leq \displaystyle \sum_{n=1}^{\infty}\Big\|K_n\Big\|^2
\Big\|(\lambda-J_n)^{-1}\Big\|^2 \ \Big\|K_n^{-1}\Big\|^2  \Big\|P_n \ph\Big\|_{\E_{\frac{1}{2}}}^2.
\end{align*}

It is easy to see that there exist two constants  $C_1, C_2>0$ such that
\begin{equation*}
\|K_n\| \leq C_1|\rho_1^n(t)|, \ \ 
\|K_n^{-1}\|\leq \frac{C_2}{|\rho_1^n|}  \quad\mbox{for
all }\;\, n \geq 1.
\end{equation*}
Now
\begin{align*}
\|(\lambda - J_n)^{-1}\|^2&= \left\|\left(\begin{matrix}
\frac{1}{\lambda-\rho_1^n}&0\\ \\ \\
0 & \frac{1}{\lambda-\rho_2^n}\\
\end{matrix}
\right)\right\|^2\\
 &\leq \frac{1}{|\lambda-\rho_1^n|^2}+
 \frac{1}{|\lambda-\rho_2^n|^2}.
\end{align*}
Define the function
$$\Theta(\lambda):=\frac{|\lambda|}{|\lambda-\rho_1^n(t)|}.$$

It is clear that $\Theta$ is
continuous and bounded on $S_\omega$.
If we take $$
C_3=\sup\left\{\frac{|\lambda|}{|\lambda-\lambda_k^n|}\; :
\; \lambda \in S_{\omega},n\geq 1\,;\; k=1,2\right\}$$it follows that
\begin{equation*}
\|(\lambda - J_n)^{-1}\| \leq
\frac{C_3}{\left|\lambda\right|},\quad   \lambda \in S_{\omega}.
\end{equation*}
Therefore, one can find a constant $K \geq 1$ such
\begin{equation*} \|R(\lambda,\a)\|_{B(\E_{\frac{1}{2}})}  \leq \frac{K}{|\lambda|}, \ \ \lambda \in S_{\omega},
\end{equation*}
and hence the operator $\a$ is sectorial on $\E_{\frac{1}{2}}$.

Since $\a$ is sectorial in $\E_{\frac{1}{2}}$, then it generates an analytic semigroup
$({\mathcal T} (\tau))_{\tau \geq 0} :=(e^{\tau \a})_{\tau\geq 0}$
on $\E_{\frac{1}{2}}$ given by
$$
e^{\tau \a}\ph=\displaystyle \sum_{n=0}^{\infty} K_n^{-1}P_n e^{\tau
J_n}P_n K_n P_n \ph.
$$

First of all, note that the semigroup $({\mathcal T} (\tau))_{\tau \geq 0}$ is hyperbolic as $\sigma(\a) \cap i \R = \emptyset$. In order words,
$\a$ satisfies an assumption similar to (H.1). 

Secondly, 
using the fact that $\a$ is an operator of compact resolvent it follows that ${\mathcal T} (\tau)$ is compact for $\tau > 0$.
 On the other hand,  we have
 \begin{eqnarray*}
 \|e^{\tau \a }\ph\|_{\E_{\frac{1}{2}}} &=&\displaystyle \sum_{n=0}^{\infty}\|K_n^{-1}P_n\|
 \|e^{\tau J_n}P_n\| \|K_n P_n\|
 \|P_n \ph\|_{\E_{\frac{1}{2}}},
 \end{eqnarray*}
 with for each  $\ph=\left(\begin{smallmatrix}\ph_1\\ \\ \\\ph_2\end{smallmatrix}\right) \in \E_{\frac{1}{2}},$
 \begin{eqnarray*}
  \|e^{\tau J_n}P_n\ph\|_{\E_{\frac{1}{2}}}^2 &=& \left\|
  \left(\begin{matrix}
   e^{\rho_1^n  \tau} E_n&0\\ \\ \\
   0&   e^{\rho_2^n \tau}E_n\\
  \end{matrix}
  \right)
  \left(
  \begin{matrix}
    \ph_1\\ \\ \\
    \ph_2\\
 \end{matrix}\right)
\right\|_{\E_{\frac{1}{2}}}^2\\
& \leq & \| e^{\rho_1^n \tau}E_n \ph_1\|_{\frac{1}{2}}^2+
\|e^{\rho_2^n\tau}E_n \ph_2\|^2\\
& \leq & e^{ \Re e (\rho_1^n ) \tau}\|\ph\|_{\E_{\frac{1}{2}}}^2.
 \end{eqnarray*}

Using Eq. (\ref{AS1}) it follows there exists $N' \geq 1$ such that
\begin{equation*}\label{eq5}
\|{\mathcal T}(\tau)\|_{B(\E_{\frac{1}{2}})} \leq N' e^{-\delta_0 \tau}, \quad \tau\geq0,
\end{equation*}
and hence ${\mathcal T}(\tau)$ is
exponentially stable, that is, $\a$ satisfies an assumption similar to (H.2) in $\E_{\frac{1}{2}}$.

{\bf Step 3}. The fact that $F$ satisfies similar assumptions as (H.4) and (H.5) is clear.

\end{proof}

\section{Example} In this section, we take $\alpha = \beta = \frac{1}{2}$. Let $\Omega \subset \R^N$ be an open bounded set with sufficiently smooth boundary $\partial \Omega$ and
let $\H = L^2(\Omega)$ be the Hilbert space of all measurable functions $\ph: \Omega \mapsto \C$ such that 
$$\|\ph\|_{L^2(\Omega)} = \Bigg(\int_{\Omega} |\ph(x)|^2 dx\Bigg)^{1/2} < \infty.$$

Here, we study the existence of pseudo-almost automorphic solutions $\ph(t,x)$ to a structurally
damped plate-like system given by (see also Chen and Triggiani \cite{CT3}, Schnaubelt and Veraar \cite{SN}, Triggiani \cite{TRIG}),
\begin{eqnarray}\label{D1000} \hspace{1.5cm} \frac{\partial^2 \ph}{\partial t^2}(t,x) -2\gamma \Delta \frac{\partial \ph}{\partial t}(t,x) + \Delta^2 \ph(t,x) &=&   \int_{-\infty}^t b(t-s) \ph (s,x) ds + f(t, \ph(t,x))\\
&+& \eta  \ph(t,x), \ \  (t,x)\in \R \times \Omega\nonumber
\end{eqnarray}
\begin{eqnarray}\label{D2000}
\Delta \ph (t,x) = \ph(t,x)  &=& 0, \ \ (t,x)\in \R \times \partial \Omega,\end{eqnarray}
where $\gamma, \eta > 0$ are constants, $b: \R \mapsto [0, \infty)$ is a measurable function, the function $f: \R \times L^2(\Omega) \mapsto L^2(\Omega)$ is pseudo-almost automorphic in $t \in \R$ uniformly in the second variable, and $\Delta$ stands for the usual Laplace operator in the space variable $x$.

Setting 
$$A  \ph =  \Delta^2 \ph \ \ \mbox{for all} \ \ \ph \in D(A) = D(\Delta^2) = \Big\{\ph \in H^4 (\Omega): \Delta \ph = \ph = 0 \ \ \mbox{on} \ \ \partial \Omega\Big\},$$
$$B \ph = A^{\frac{1}{2}}\ph = - \Delta \ph, \ \ \forall \ph \in D(B) = H_0^1(\Omega) \cap H^2(\Omega),$$
$$C(t)\ph = b(t) \ph \ \ \mbox{for all} \ \ \ph \in D(C(t)) = D(\Delta),$$ and $f: \R \times \Big(H_0^1(\Omega) \cap H^2(\Omega)\Big) \mapsto L^2(\Omega)$, one can easily see that Eq. (\ref{PRRR}) is exactly
the structurally
damped plate-like system formulated in 
Eqs. (\ref{D1000})-(\ref{D2000}).

Here $\E_{\frac{1}{2}} = D(A^{\frac{1}{2}}) \times L^2(\Omega) = \Big(H_0^1(\Omega) \cap H^2(\Omega)\Big) \times L^2(\Omega)$ and it is equipped with the inner product defined by
$$\Bigg(\left(\begin{smallmatrix}\displaystyle \ph_1 \\ \\ \\ \displaystyle \ph_2\end{smallmatrix}\right), \left(\begin{smallmatrix}\displaystyle \psi_1 \\ \\ \\  \displaystyle \psi_2\end{smallmatrix}\right) \Bigg)_{\E_{\frac{1}{2}}} := \int_{\Omega} \Delta \ph_1 \overline{\Delta \psi_1} dx + \int_{\Omega} \ph_2 \overline{\psi_2} dx$$
for all $\ph_1, \psi_1 \in H_0^1(\Omega) \cap H^2(\Omega)$ and $\ph_2, \psi_2 \in L^2(\Omega)$. 

Similarly, $D(A^{\frac{1}{2}}) = H_0^1(\Omega) \cap H^2(\Omega)$ is equipped with the norm defined by
$$\|\ph\|_{\frac{1}{2}} = \|A^{\frac{1}{2}}\ph\|_{L^2(\Omega)} := \Big(\int_{\Omega} |\Delta \ph|^2dx \Big)^{\frac{1}{2}}$$
for all $\ph  \in H_0^1(\Omega) \cap H^2(\Omega)$.

Clearly, $-A_\eta = - (\Delta^2 + \eta I)$ is a sectorial operator on $L^2(\Omega)$ and let $(T(t))_{t \geq 0}$ be the analytic semigroup associated with it. It is well-known that the semigroup $T(t)$ is not only compact for $t > 0$ but also is exponentially stable as
$$\|T(t)\| \leq  e^{-\eta t}$$ for all $t \geq 0$.

Using the fact the Laplace operator $\Delta$ with domain $D(\Delta) = H^2(\Omega) \cap H_0^1(\Omega)$ is invertible in $L^2(\Omega)$ it follows that 
\begin{eqnarray*}
\|\ph\|_{L^2(\Omega)} &=& \|\Delta^{-1} \Delta \ph\|_{L^2(\Omega)} \\
&\leq& \|\Delta^{-1}\|_{B(L^2(\Omega))} \,.\,\|\Delta \ph\|_{L^2(\Omega)} \\
&=& \|\Delta^{-1}\|_{B(L^2(\Omega))} \,.\,\|A^{\frac{1}{2}}\ph\|_{L^2(\Omega)}\\
&=& \|\Delta^{-1}\|_{B(L^2(\Omega))} \,.\,\|\ph\|_{\frac{1}{2}}\\
\end{eqnarray*}for all $\ph \in H^2(\Omega) \cap H_0^1(\Omega)$.

If $b \in L^1(\R, (0, \infty))$, then using the previous inequality it follows that 
\begin{eqnarray*}
\|\c(t) \Phi\|_{\E} &=&  b(t) \left\|\ph\right\|_{L^2(\Omega)}\\
&\leq& b(t) \|\Delta^{-1}\|_{B(L^2(\Omega))} \,.\,\|\ph\|_{\frac{1}{2}} \\
&\leq& b(t) \|\Delta^{-1}\|_{B(L^2(\Omega))}\,.\, \|\Phi\|_{\E_{\frac{1}{2}}}\\
\end{eqnarray*}
for all $\displaystyle \Phi = \left(\begin{smallmatrix}\displaystyle \ph \\ \\ \\ \displaystyle \psi\end{smallmatrix}\right) \in \E_{\frac{1}{2}}$ and $t \in \R$.

This setting requires the following assumptions,
\begin{enumerate}
\item[(H.7)] Eq. (\ref{AS1}) holds.
\item[(H.8)] The function $b: \R \mapsto [0, \infty)$ belongs to $L^1(\R, (0, \infty))$ with $$\displaystyle \int_{-\infty}^\infty b(s) ds \leq \frac{1}{2\|\Delta^{-1}\|_{B(L^2(\Omega))} d(\frac{1}{2})},$$
where $d (\frac{1}{2}) := M(\frac{1}{2}) \sqrt{\displaystyle \frac{2\pi}{\eta}}$.
\end{enumerate}

In view of the above, it is clear that assumptions (H.1)-(H.2)-(H.3)-(H.6) are fulfilled. Therefore, using Corollary \ref{ABB}, we obtain the following theorem.

\begin{theorem}\label{ABC}
Under assumptions {\rm (H.4)--(H.5)--(H.7)--(H.8)}, then the system  Eqs. (\ref{D1000})-(\ref{D2000}) has at least one pseudo-almost automorphic mild solution.
\end{theorem}

\bibliographystyle{amsplain}

\end{document}